\documentclass[12pt]{article}
\usepackage{xcolor}
\usepackage{circuitikz}
\usepackage{ifthen}

\usepackage{macros}

\mathtoolsset{showonlyrefs}

\begin{document}
\title{Expected Length of the Longest Common Subsequence of Multiple Strings}
\author{Ray Li\thanks{Math \& CS Department, Santa Clara University. Email: \url{rli6@scu.edu}.} , William Ren\thanks{Math \& CS Department, Santa Clara University. Email: \url{wren@scu.edu}.} , Yiran
Wen\thanks{Math \& CS Department, Santa Clara University. Email: \url{ywen@scu.edu}.}}
\date{\today}
\maketitle

\begin{abstract}
  We study the generalized Chvátal-Sankoff constant $\gamma_{k,d}$, which represents the normalized expected length of the longest common subsequence (LCS) of $d$ independent uniformly random strings over an alphabet of size $k$. We derive asymptotically tight bounds for $\gamma_{2,d}$, establishing that $\gamma_{2,d} = \frac{1}{2} + \Theta\left(\frac{1}{\sqrt{d}}\right)$. We also derive asymptotically near-optimal bounds on $\gamma_{k,d}$ for $d\ge \Omega(\log k)$. 
\end{abstract}

\section{Introduction}
The Longest Common Subsequence (LCS) is a fundamental measure of the similarity of two or more strings that is important in theory and practice.
A \emph{subsequence} of a string is obtained by removing zero or more characters, and  the \emph{Longest Common Subsequence} (LCS) of $d$ strings $X^1,\dots,X^d$ is the longest subsequence that occurs in all of $X^1,\dots,X^d$.
For $d$ strings $X^1, \dots,X^d$, we let $\LCS(X^1,\dots,X^d)$ denote the length of their LCS. For example $\LCS(0011,0101) = 3$.
Computing the LCS is a textbook application of dynamic programming in computer science \cite{wagner1974string}, and the algorithm has many applications from text processing, to linguistics, to computational biology. As one example, the linux \verb|diff| tool uses a variation of the LCS algorithm.

Chvátal and Sankoff \cite{chvatal1975longest} showed that as $n$ approaches infinity, the normalized expected length of the LCS of two independent uniformly random binary strings converges to a constant.
This limit is known as the Chvátal--Sankoff constant,
\begin{align}
    \gamma\defeq \lim_{n\to\infty}\frac{\E_{X^1,X^2\sim \{0,1\}^n}[\LCS(X^1,X^2)]}{n}
\end{align}
where the expectation is over independent uniformly random binary strings $X^1, X^2$.
Determining $\gamma$ is an open question with a rich history \cite{chvatal1975longest,deken1979some,steele1986longest, paterson1994longest,boutet1999longest,baeza1999new,bundschuh2001analysis,lueker2003,lueker2009improved,bukh2022length,heineman2024improved}. 
Currently the best bounds are roughly 
$ 0.792665\le \gamma \le 0.826280$ \cite{heineman2024improved,lueker2009improved}. 

Table~\ref{tab:bounds} presents a summary of key works that have contributed to establishing bounds on the Chvátal-Sankoff constant. Only some studies offer rigorously proven bounds, while others present estimates.
\begin{table}[h!]
\centering
\caption{History of Bounds and Estimates for the Chvátal-Sankoff Constant, \( \gamma \)}
\label{tab:bounds}
\begin{tabular}{|l|c|c|c|}
\hline
Authors & Year & Proven Bounds & Estimates \\
\hline
Chvátal and Sankoff~\cite{chvatal1975longest} & 1975 & \( 0.697844 \leq \gamma \leq 0.866595 \) & \( \gamma \approx 0.8082 \) \\
Deken~\cite{deken1979some} & 1979 & \( 0.7615 \leq \gamma \leq 0.8575 \) & \\
Steele~\cite{steele1986longest} & 1986 & & \( \gamma \approx 0.8284 \) \\
Dančík and Paterson~\cite{dancik1995longest} & 1995 & \( 0.77391 \leq \gamma \leq 0.83763 \) & \\
Boutet de Monvel~\cite{boutet1999longest} & 1999 & & \( \gamma \approx 0.812282 \) \\
Baeza-Yates et al.~\cite{baeza1999new} & 1999 & & \( \gamma \approx 0.8118 \) \\
Bundschuh~\cite{bundschuh2001analysis} & 2001 & & \( \gamma \approx 0.812653 \) \\
Lueker~\cite{lueker2009improved} & 2009 & \( 0.788071 \leq \gamma \leq 0.82628 \) & \\
Bukh and Cox~\cite{bukh2022length} & 2022 & & \( \gamma \approx 0.8122 \) \\
Heineman et al.  \cite{heineman2024improved} & 2024 & \(  
0.792665 \leq \gamma  \) & \\
\hline
\end{tabular}
\end{table}

There are two natural ways to generalize the Chvátal-Sankoff problem: (1) increase the alphabet size and (2) increase the number of strings. In this way, we may generalize the Chvátal and Sankoff constant by asking for $\gamma_{k,d}$, the (normalized) expected longest common subsequence of $d$ independent uniformly random strings over a size-$k$ alphabet. Formally, let
    \begin{align}
    \gamma_{k,d} = \lim_{n \to \infty } \frac{\E_{X^1,\dots,X^d\sim [k]^n}[\LCS(X^1,\cdots, X^d)]}{n}
    \end{align}
where the expectation is over independent uniformly random strings $X^1,\cdots X^d\sim[k]^n$, where $[k]=\{1,\dots,k\}$. 
By definition, $\gamma = \gamma_{2,2}$.

The generalization to larger alphabet size $k$ is well studied and well understood. This line of work \cite{deken1979some,dancik1994expected,paterson1994longest,baeza1999new} culminated in a beautiful result that $\gamma_{k,2} \to \frac{2}{\sqrt{k}}$ as $k\to\infty$ \cite{kiwi2005expected}, answering a conjecture of Sankoff and Mainville \cite{sankoff1983common}.

We study the generalization to more strings $d$, which is also an important question. Mathematically it is a fundamental generalization of the Chvatal-Sankoff constant. In computer science, it is intimately connected to error-correcting codes list-decodable against deletions \cite{kash2011zero} (see also \cite{GuruswamiW14,guruswami2020optimally,guruswami2022zero}). Specifically, $1-\gamma_{k,d}$ is the maximum fraction of deletions that a positive rate random code can list-decode against with list-size $d-1$. This connection follows from a generalization of a martingale concentration argument shown in \cite{kash2011zero}, and for completeness, we show the connection in Appendix~\ref{app:deletion}. 

Several works have previously considered the generalizing the number of strings $d$, but less is known than for the larger-alphabet generalization.
Jiang and Li \cite{jiang1995approximation} showed that when $d=n$, the expected LCS of $d$ strings is roughly  $\frac{n}{k}$. Dancik \cite{danvcik1998common} showed that, for fixed $d$, $\gamma_{k,d} = \frac{c}{k^{1-1/d}}$ for some constant $c\in[1,e]$, disproving a conjecture of Steele~\cite{steele1986longest} that $\gamma_{k,d}=\gamma_{k,2}^{d-1}$.
Kiwi and Soto~\cite{kiwi2008} established numerical bounds on $\gamma_{k,d}$ for small values of $k$ and $d$. For example, they obtain bounds on $\gamma_{k,d}$ up to $d=14$ for binary alphabet, and up to alphabet size $k=10$ for $d=3$ strings. 
A recent work of Heineman et al. \cite{heineman2024improved} improves  upon \cite{kiwi2008} and establishes stronger numerical bounds. 

\paragraph{Our Contributions.}
We give tight asymptotic bounds on the binary  Chvátal-Sankoff constant as the number of strings increases, showing $\gamma_{2,d} = \frac{1}{2}+\Theta(\frac{1}{\sqrt{d}})$.
\begin{theorem}
There exists constants $0<c_1<c_2$ such that, for all integers $d\ge 2$ we have
\begin{align}
    \frac{1}{2} + \frac{c_1}{\sqrt{d}}\le \gamma_{2,d} \le \frac{1}{2} + \frac{c_2}{\sqrt{d}}
\end{align}
    \label{thm:main}
\end{theorem}
Our main contribution is the lower bound, which combines a technique of Lueker \cite{lueker2009improved} and Kiwi and Soto \cite{kiwi2008} with a greedy matching strategy.
Our upper bound follows from a counting argument of Guruswami and Wang \cite{GuruswamiW14}, who studied codes for list-decoding deletions.

We also give bounds that are asymptotically near-optimal bounds for larger alphabets. 

\begin{theorem}
There exists constants $c_0,c_1,c_2> 0$ such that, for all integers $d$ and $k$ be integers with $d\ge c_0\log k$, we have
\begin{align}
    \frac{1}{k}\left(1 + \frac{c_1}{\sqrt{d}}\right)\le \gamma_{k,d} \le \frac{1}{k}\left(1 + c_2\sqrt{\frac{\log k}{d}}\right)
\end{align}
\label{thm:k}
\end{theorem}

The lower bound of Theorem~\ref{thm:k} follows from Theorem~\ref{thm:main} by noting that $\gamma_{k,d} \ge \frac{2}{k}\gamma_{2,d}$: random $k$-ary strings of length $n$ typically have binary subsequences of length roughly $\frac{2}{k}n$. (See Appendix~\ref{app:k}).
The upper bound again follows from a counting argument of Gurusuwami and Wang \cite{GuruswamiW14}.

\paragraph{Organization of the paper.}

In Section~\ref{sec:proofsketch}, we illustrate the ideas in our proof by sketching the proof in the binary case, $k=2$. In Section~\ref{sec:prelims}, we present preliminaries for the proofs. In Section~\ref{sec:proof}, we prove Theorem~\ref{thm:main}.

\section{Proof Overview}
\label{sec:proofsketch}

We now sketch the proof of Theorem~\ref{thm:main} binary case, $k=2$.
We start with the lower bound.

\subsection{The Kiwi-Soto Algorithm}
Our first step is to reduce the generalized Chvátal–Sankoff $\gamma_{k,d}$ problem to estimating the expected \emph{Diagonal LCS}.
This approach was considered by Lueker \cite{lueker2003}, who focused on the two-string case ($d=2$) and obtained numerical lower bounds.
It was then generalized by Kiwi and Soto \cite{kiwi2008} (see also \cite{heineman2024improved}) to obtain numerical lower bounds for more strings $d\ge 3$.
We use the same technique to find lower bounds for any number of strings $d$.

Let \(A_1,\ldots, A_d\) be a collection of \(d\) finite binary strings.
Let \(X_1, \ldots, X_d\) be a collection of $d$ independent uniformly random binary strings of length $n$.
For a string $X$, let \(X[i]\) denote the sub-string formed by the first $i$ characters of string $X$.
Lueker (for $d=2$) and Kiwi and Soto (for all $d$) define,
\begin{align}
W_n(A_1, \ldots, A_d) = \mathbb{E}_{X_1,\dots,X_d} \left[ \max_{\substack{i_1 + \cdots + i_d = n}} \LCS(A_1 X_1[1..i_1], \ldots, A_d X_d[1..i_d]) \right].\label{eq:ks}
\end{align}
and show
\begin{align}
\gamma_{2, d} = \lim_{n \to \infty} \frac{W_{nd}(A_1,\dots,A_d)}{n}.
\end{align}
for all fixed strings $A_1,\dots,A_d$.
Leuker and Kiwi and Soto combine this result with a dynamic programming approach to find numerical lower bounds on $\lim_{n\to\infty} \frac{W_{nd}}{n}$, and thus $\gamma_{2,d}$ (and, more generally, $\gamma_{k,d}$).

We take $A_1,\dots,A_d$ to be the empty string.
Define the expected Diagonal LCS as 
\begin{align}
    W_n \defeq \mathbb{E} \left[ \max_{\substack{i_1 + \cdots + i_d = n}} \LCS(X_1[1..i_1], \ldots, X_d[1..i_d]) \right] = W_n( \lambda,\cdots,\lambda),
    \label{eq:wn}
\end{align}
where $\lambda$ denotes the empty string.
By \eqref{eq:ks}, we have
\begin{align}
    \gamma_{2, d} = \lim_{n \to \infty} \frac{W_{nd}}{n}.
    \label{eq:diag}
\end{align}
Intuitively, \eqref{eq:diag} is true because the maximum in \eqref{eq:wn} is obtained when $i_1,i_2,\dots,i_d$ are all roughly $n/d$, so $W_n$ approaches to the expected LCS of $d$ strings of length $n/d$.

\subsection{The Binary Lower Bound and Matching Scheme}
Now we find a lower bound for $\lim_{n\to\infty}\frac{W_{nd}}{n}$, and thus \(\gamma_{2,d}\).
To do this, we find a common subsequence between $d$ random strings  by defining a matching strategy that finds the bits of the common subsequence one at a time.
We track the number of bits we ``consume'' across the $d$ strings, per 1 matched LCS bit.
We show that our greedy matching consumes on average $2d-\Theta(\sqrt{d})$ bits per 1 matched LCS bit, which on average, gives us $\frac{nd}{2d-\Theta(\sqrt{d})} = n(\frac{1}{2} + \Theta(\frac{1}{\sqrt{d}}))$ LCS bits for $nd$ symbols consumed.
These estimates suggest $W_{nd} \ge n(\frac{1}{2} + \Theta(\frac{1}{\sqrt{d}}))$, and thus $\gamma_{2,d} \ge \frac{1}{2} + \Theta(\frac{1}{\sqrt{d}})$, and we then prove this estimate.

We now describe the matching strategy. We match the LCS bit by bit, revealing the random bits as we need them; importantly, because the bits are independently random, we can reveal them in any desired order.
For each LCS bit, we reveal the next bit in each of the $d$ strings. We then take the next LCS bit to be the majority bit, say 0, and find the next 0 in each of the $d$ strings.
The number of bits consumed can be described by a process of repeatedly flipping $d$ fair coins until all coins show the same face. 
We first flip all $d$ coins. 
We keep re-flipping all the coins in the minority until they show the majority face.
For example, suppose we have flipped the $d$ coins and heads appears \( \left\lceil \frac{d}{2} \right\rceil\) times. Then we repeatedly re-flip the \(\left\lfloor \frac{d}{2} \right\rfloor\) coins that landed tails, until each shows heads.
We let $Z$ be the random variable denoting the total number of coin flips, or, equivalently, the total number of bits consumed per 1 LCS bit.

To analyze the expected number of flips, we first consider the random variable $Y$, 
the number of coins in the minority after the first $d$ flips.
In the binary case, it is not hard to compute the expectation of $Y$ explicitly. For example, when $d$ is even we have:
\[
\mathbb{E}[Y] 
= \frac{1}{2^{d}}\left(\sum_{i=0}^{d/2-1} \binom{d}{i} \cdot 2i + \binom{d}{d/2} \cdot (d/2)\right)
= \frac{d}{2^d}\left(2^{d-1} - \binom{d-1}{d/2}\right) 
\approx \frac{d}{2} - \Theta(\sqrt{d})
\]
and a similar computation holds when $d$ is odd.
Intuitively, the estimate $\E[Y]=\frac{d}{2}-\Theta(\sqrt{d})$ makes sense because $Y=d/2 - |d/2-h|$ where $h$ is the number of heads.
The standard deviation of $h$ is $\Theta(\sqrt{d})$, so we ``expect'' $|d/2-h|$ to be $\Theta(\sqrt{d})$, and thus $Y$ to be $d/2-\Theta(\sqrt{d})$.

Now that we have a handle on $Y$, we can study $Z$, the total number of bits consumed for 1 LCS bit.
The number of reflips of each minority coin is a geometric random variable with $p=1/2$. Thus, the expected number of reflips of each minority coin is 2.
Taking into account the conditional expectations, we can show that the expected total number of reflips of minority coins is thus $2\cdot \E[Y] =d - \Theta(\sqrt{d})$.
Adding on the $d$ initial flips, we have
\[
    \E[Z] = d + 2\E[Y] = 2d- \Theta(\sqrt{d}).
\]

\begin{figure}[t]
\centering
\resizebox{1\textwidth}{!}{%
\begin{circuitikz}[>=latex]
\tikzstyle{every node}=[font=\LARGE]
\foreach \i in {0,...,6}{
    \begin{scope}[shift={(0,-\i*1.5)}]

        \draw  (1.75,18.75) rectangle (28.75,17.75);
        \draw (2.75,18.75) -- (2.75,17.75);
        \draw (3.75,18.75) -- (3.75,17.75);
        \draw (4.75,18.75) -- (4.75,17.75);
        \draw (5.75,18.75) -- (5.75,18);
        \draw (5.75,18) -- (5.75,17.75);

        \node at (6.75,18) {.....};

        \ifnum\i<4
          
            \node[text=red, font=\bfseries] at (2.25,18.25) {1};

        \else
           
            \node at (2.25,18.25) {0};
            \node at (3.25,18.25) {0};
            \node[text=red, font=\bfseries] at (4.25,18.25) {1};

        \fi
    \end{scope}
}
\end{circuitikz}
}%
\caption{Our matching strategy for $d=7$ random binary strings. Because all bits are independent, we can reveal the randomness in any order. We generate 7 random bits. Suppose, as illustrated, 4 bits are a $1$, and $Y=3$ are a 0. We reveal more bits in the strings with 0s until we see 1s. Here, in total, to get 1 LCS bit, we revealed the randomness from $Z=13$ bits across the 7 strings.\\
}
\label{fig:Matching Strategy}
\end{figure}
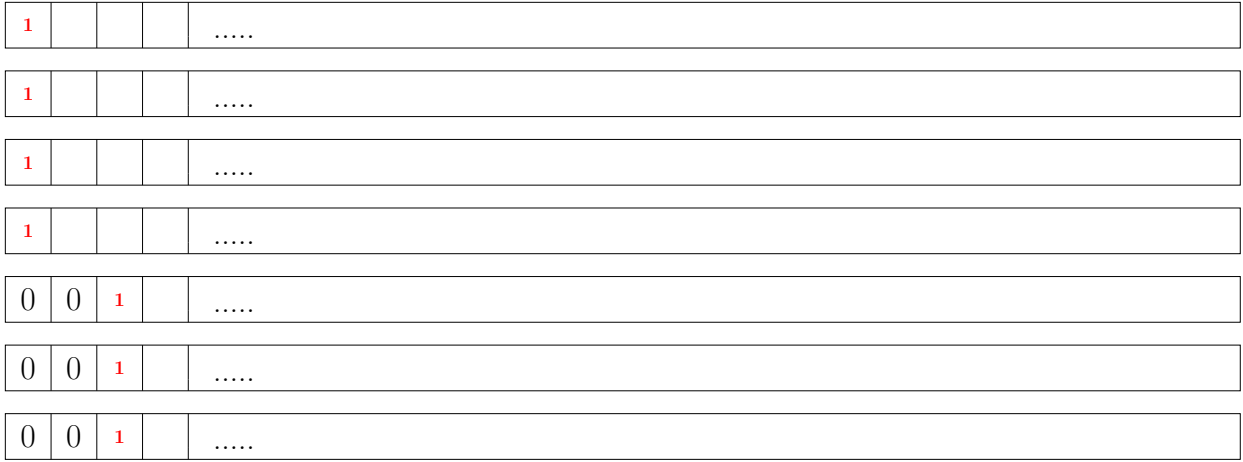

This shows (modulo some details) that our greedy matching strategy consumes $2d-\Theta(\sqrt{d})$ bits per 1 matched bit. Our back-of-the-envelope calculation suggests that, because we have $nd$ bits to consume across the $d$ strings, and we consume an average of $2d-\Theta(\sqrt{d})$ bits per matched bit, we expect to find a common subsequence of length at least $\frac{nd}{2d-\Theta(\sqrt{d})} = n(\frac{1}{2} + \Theta(\frac{1}{\sqrt{d}}))$, as desired.

However, we have to work harder to formally justify this.
Let $Z_1,Z_2,\dots$ be the random variables where $Z_i$ denotes the number of bits needed to consume to match the $i$th bit with our matching strategy.
By carefully choosing the order in which we reveal our bits, we have that $Z_1,Z_2,\dots$ are mutually independent.
Further, the $Z_i$ are identically distributed as $Z$, and thus have expectation $2d-\Theta(\sqrt{d})$.
The number of bits we matched by our strategy is the largest $L$ such that $Z_1+\cdots+Z_L\le nd$.
Importantly, because we work with Diagonal LCS, we do not need to worry that we use a different number of bits in different strings.
To show the expected number bits matched is close to our estimate, we show that, for $L_0=\frac{nd}{\E[Z]} ( 1-o(1))$, 
\begin{align}
    \Pr[Z_1+\cdots+Z_{L_0} \le nd] > 1-o(1).
    \label{eq:zsum}
\end{align}
We cannot use a standard concentration inequality because the $Z_i$ are unbounded. 
However, each $Z_i$ is the sum of at most $d$ geometric random variables.
Thus, setting $Z_i' \defeq \min(Z_i, O_d(\log n))$, we have, with high probability, $Z_i'=Z_i$ for all $i$.
We then use concentration inequalities to show $Z_1'+\cdots+Z_{L_0}'\le nd$ with high probability, and then \eqref{eq:zsum} holds.
Thus, the expected number of bits matched is at least $L_0\cdot (1-o(1)) \ge n(\frac{1}{2}+\Omega(\frac{1}{\sqrt{d}})$.
Hence, we can conclude our bound
\begin{align}
    \gamma_{2,d} \ge \frac{1}{2} + \Omega\left(\frac{1}{\sqrt{d}}\right).
\end{align}

\subsection{The Binary Upper Bound}

The upper bound follows from a counting argument.
Guruswami and Wang \cite[Lemma 2.3]{GuruswamiW14} (Lemma~\ref{lem:gw14}) bound the number of supersequences of any string of length $\ell>\frac{n}{k}$.
By applying this bound and carefully tracking the lower-order terms, we show that for $\Pr[\LCS(X_1,\dots,X_d)\ge \ell]$ is exponentially small for $\ell = \frac{n}{k}(1+\Theta(\sqrt{\frac{\log k}{d}}))$.
Our bound on the expectation follows.

\section{Preliminaries}
\label{sec:prelims}

Throughout $\log$ is base 2 unless otherwise specified, and $\ln$ is log base-$e$.
We use the following result.
\begin{lemma}\label{lem:ConditionalExpectation}
Let $Y,W_1,W_2,\dots$ be independent random variables supported on nonnegative integers where $W_1,W_2,\dots$ are identically distributed. Define \( W = W_1 + W_2 + \cdots + W_Y \). Then, 
\begin{equation}\label{eq:ConditionalExpectation}
\E[W] = \E[Y] \E[W_1].
\end{equation}
\end{lemma}

\begin{proof}
Using the law of total expectation, we have:
\[
\E[W] = \E[\E[W \mid Y]].
\]
By the linearity of expectation, given the value of \(Y\), we obtain:
\[
\E[W \mid Y] = \E[W_1] + \E[W_2] + \cdots + \E[W_Y] = Y \E[W_1].
\]
Substituting this into the equation above, we get:
\[
\E[W] = \E[Y] \E[W_1]. \qedhere
\]
\end{proof}

We also use Hoeffding's Inequality.
\begin{lemma}[Hoeffding]\label{lem:Hoeffding}
 Let \( X_1, X_2, \dots, X_n \) be independent random variables such that \( X_i \in [a_i, b_i] \) almost surely. Then, for any \( t > 0 \),
 \begin{align}\label{eq:HoeffdingIneq}
   \mathbb{P}\left( \sum_{i=1}^n X_i - \mathbb{E}\left[ \sum_{i=1}^n X_i \right] \geq t \right) \leq \exp\left( -\frac{2 t^2}{\sum_{i=1}^n (b_i - a_i)^2} \right).
 \end{align}
\end{lemma}

For $p \in (0,1)$, we define the $q$-ary entropy as
\[
H_q(p)=p\log_{q}(q-1)-p\log_{q}(p)-(1-p)\log_{q}(1-p)
\]
where \( h(p) \) is the binary entropy function.
We use the following well known estimate on binomial terms.
\begin{lemma}[see, for example, Proposition 3.3.1 of \cite{guruswami2019essential}]
We have    \label{lem:binom-apx}
\begin{align}
    \binom{m}{pm}(q-1)^{pm} \le q^{H_q(p)m}.
\end{align}
\end{lemma}

We use the following estimate for $k$-ary entropy.
\begin{lemma}[see, for example, Proposition 3.3.5 of \cite{guruswami2019essential}]
For small enough $\epsilon \in (0,\frac{1}{k})$: 
\[
H_k \left( 1 - \frac{1}{k} - \epsilon \right) \leq 1 - c_k\cdot  \epsilon^2.
\]
for constant $c_k\ge\frac{k^2}{4(k-1)\ln k}\ge \frac{k}{4 \ln k}$. 
\label{lem:h-apx}
\end{lemma}

As described in Section~\ref{sec:proofsketch}, define the expected diagonal LCS
\begin{align}
W_n  \defeq  \E \left[\max_{i_1+\cdots+i_d=n} \LCS(X^1[1\cdots i_1], \cdots, X^d[1\cdots i_d]) \right ] 
\label{eq:prefix}
\end{align}
where the randomness is over uniformly random infinite binary strings $X^1,\dots,X^d$.
The following lemma shows that the diagonal LCS equals the expected LCS up to lower order terms.
\begin{lemma}[\cite{kiwi2008}]
\label{lem:ks}
We have
\[
\gamma_{k,d}=\lim_{n \to \infty} \frac{W_{nd}}{n}
\]
\label{eq:kiwisoto}
\end{lemma}

\section{Full Proof of the $k$-ary LCS}
\label{sec:proof}
\subsection{Theorem~\ref{thm:main}, lower bound}

\begin{proof}[Proof of Theorem~\ref{thm:main}, lower bound]
By Lemma~\ref{lem:ks}, it suffices to show that there is an absolute constant $c_1>0$ such that, for sufficiently large $n$,
\begin{align}
    W_{nd} = \E \left[\max_{i_1+\cdots+i_d=nd} \LCS(X^1[1\cdots i_1], \cdots, X^d[1\cdots i_d]) \right ]  \ge n\left(\frac{1}{2} + \frac{c_1}{\sqrt{d}}\right)    
\end{align}

We now present our greedy matching strategy for finding a long ``diagonal'' common subsequence --- a common subsequence of $X^1[1\cdots i_1], \cdots, X^d[1\cdots i_d]$ for $i_1+\cdots+i_d=nd$.
Given $d$ random infinite strings $X^1,\dots,X^d$, we find a  the LCS bit by bit, revealing the random bits of $X^1,\dots,X^d$ as we need them.
Importantly, because the bits are independently random, we can reveal them in any desired order. Use the following process:
\begin{enumerate}  
    \item Initialize a string $s$ to the empty string, representing our common subsequence of $X^1,\dots,X^d$.
    \item Repeat the following
    \begin{enumerate}
        \item Reveal the next unrevealed bit $b_1,\dots,b_d$ in each of $X^1,\dots,X^d$.
        \item Let $b$ be the majority bit among these $d$ bits. 
        \item For each string $X^j$ that did not reveal the majority bit ($b_i\neq b$), reveal bits of $X^j$ until we reveal a bit equal to $b$. 
        \item If number of revealed bits is at most $nd$, append $b$ to $s$, else exit.
        \end{enumerate}
\end{enumerate}
See Figure~\ref{fig:Matching Strategy} for an illustration of this process.
The length of the subsequence we find is the number of times we successfully complete the loop.

Let $Y$ denote the random variable that denotes the number of minority bits among $d$ uniformly random bits.
Let $Z$ denote the random variable that first samples $Y$, and is set to $d + W_1+\cdots+W_Y$, where $W_1,\dots,W_Y$ are independent geometric random variables with probability 1/2.
Because the bits are independent, the number of bits revealed in each iteration of the loop is distributed as $Z$.
Thus, letting $Z_1,Z_2,\cdots,$ be i.i.d random variables distribute as $Z$, the length of our LCS is distributed as
\begin{align}
    L_{greedy} \defeq \max(L: Z_1+\cdots+Z_L\le nd)
    \label{eq:Lgreedy}
\end{align}
We wish to lower bound $\E[L_{greedy}]$.

We start by analyzing the expectations of $Y$ and $Z$. Explicit calculations yield that, for all $d$, 
\begin{align}
    \E[Y] \le \frac{d}{2} - c\sqrt{d} 
    \label{eq:y}
\end{align}
for some absolute constant $c>0$. To see this, note that for $d$ even,
\begin{align}
\mathbb{E}[Y] 
&= \frac{1}{2^{d}}\left(\sum_{i=0}^{d/2-1} \binom{d}{i} \cdot 2i + \binom{d}{d/2} \cdot (d/2)\right) \nonumber\\
&= \frac{1}{2^{d}}\left(\sum_{i=0}^{d/2-1} 2d\cdot \binom{d-1}{i-1}  + d\cdot \binom{d-1}{d/2-1}\right) \nonumber\\
&= \frac{d}{2^d}\left(2^{d-1} - \binom{d-1}{d/2}\right) \nonumber\\
&\le \frac{d}{2} - c\sqrt{d}
\end{align}
and for $d$ odd,
\begin{align}
\mathbb{E}[Y] 
&= \frac{1}{2^{d}}\left(\sum_{i=0}^{(d-1)/2} \binom{d}{i} \cdot 2i\right) \nonumber\\
&= \frac{1}{2^{d}}\left(\sum_{i=0}^{(d-1)/2} 2d\cdot\binom{d-1}{i-1}\right) \nonumber\\
&= \frac{d}{2^d}\left(2^{d-1} - \binom{d-1}{(d-1)/2}\right)  \nonumber\\
&\le \frac{d}{2} - c\sqrt{d}
\end{align}
where $c>0$ is some absolute constant.
Thus, \eqref{eq:y} holds.
By Lemma~\ref{lem:ConditionalExpectation} we have $\E[Z]\le d + 2\E[Y] = d-2c\sqrt{d}$.

Let \( L_0 = \frac{nd}{\mathbb{E}[Z]} \left(1 - \gamma\right) \) for $\gamma=\frac{1}{100\log n}$.
We show that the sum \( \sum_{i=1}^{L_0} Z_i \) is less than \( nd \) with very high probability, so that $L_{greedy}\ge L_0$ with very high probability.
This follows from concentration inequalities, but we cannot apply the inequalities directly because our random variables $Z_i$ are unbounded.
Define truncated variables \( Z_i' = \min(Z_i, T) \) for $T=100 d \log n$, so that each $Z_i'$ is in $[0,T]$.

We show that all $Z_i=Z_i'$ with high probability.
In step 2(c), for each $X^j$, we see the correct bit $b$ within $99\log n$ steps with probability at least $1-\frac{1}{n^{99}}$.
By the union bound, this happens for all $j=1,\dots,d$ with probability at least $1-\frac{d}{n^{99}}$, in which case $Z_i\le d+99d\log n < T$ and $Z_i=Z_i'$.
Thus, union bounding over $i=1,\dots,L_0$, we have 
\begin{align}
    \Pr[\text{$Z_i=Z_i'$ for all $i=1,\dots,L_0$}]
    \ge 1-nd\cdot \frac{d}{n^{99}}\ge 1-\frac{1}{n^{97}}.
    \label{eq:zi}
\end{align}

Since $Z_1',\dots, Z_{L_0}'$ are independent, Hoeffding's inequality (Lemma~\ref{lem:Hoeffding}) implies
\[
\mathbb{P}\left[\sum_{i=1}^{L_0} Z_i' > \mathbb{E}\left[\sum_{i=1}^{L_0} Z_i'\right] + t \right] \leq \exp\left(-\frac{2t^2}{\sum_{i=1}^{L_0} T^2}\right),
\]
where $t = nd - \mathbb{E}\left[\sum_{i=1}^{L_0} Z_i'\right] = nd-L_0\cdot \E[Z'] \ge \gamma nd$. 
Substituting $t$ gives,
\begin{align}
\mathbb{P}\left[Z_1' + \dots + Z_{L_0}' > nd\right] \leq \exp\left(-\frac{\gamma^2n^2d^2}{T^2\cdot L_0}\right) \le \exp\left(-\Omega_{d}\left(\frac{n}{\log^3n}\right)\right).
\label{eq:concentrate}
\end{align}
Combining \eqref{eq:zi} and \eqref{eq:concentrate}, we have, for sufficiently large $n$
\[
\mathbb{P}\left[Z_1 + \dots + Z_{L_0} > nd\right]\le \mathbb{P}\left[Z_1' + \dots + Z_{L_0}' > nd\right] + \Pr[Z_i\neq Z_i' \text{ for some $i$}] \le \frac{2}{n^{97}}.
\]
Finally, the expected LCS length after \( nd \) bits is:
\begin{align}
\E[L_{greedy}]
&\ge\mathbb{E}\left[\max(L:Z_1 + \dots + Z_{L} \leq nd)\right] \nonumber\\
&\ge
L_0\cdot \Pr[Z_1+\cdots+Z_{L_0}\le nd]\nonumber\\
&\geq \frac{nd}{\mathbb{E}[Z]}\left(1-\gamma\right)\cdot \left(1-\frac{2}{n^{97}}\right) \nonumber\\
&\ge n\left(\frac{1}{2} + \frac{c_1}{\sqrt{d}}\right)
\end{align}
for some absolute constant $c_1>0$. Hence,
\begin{align}
\gamma_{2,d} 
&= \lim_{n \to \infty} \frac{W_{nd}}{n} \nonumber\\
&= \lim_{n \to \infty}\frac{\E[L_{greedy}]}{n} \nonumber\\
&\ge \frac{1}{2} \left( 1 + \frac{c_1}{\sqrt{d}} \right).\qedhere\nonumber
\end{align}

\end{proof}

\subsection{Theorem~\ref{thm:main}, upper bound}

We use the following lemma from \cite{GuruswamiW14} that counts superstrings of a string of a given length.
\begin{lemma}[Lemma 2.3 of \cite{GuruswamiW14}]
For any string $w$ of length $\ell>\frac{n}{k}$, the number of strings of length $n$ with $w$ as a subsequence is at most\footnote{The result in \cite{GuruswamiW14} is stated for $\ell>(1-1/k)n$, and states that there are at most $\sum_{t=\ell}^n\binom{t-1}{\ell-1}k^{n-t}(k-1)^{t-\ell}$ subsequences. However, this bound comes from a counting argument and actually holds for all $\ell$. For $\ell > n/k$, the summands increase with $t$, so bounding each summand by the $t=n$ summand gives the bound stated here.}
\label{lem:gw14}
\[
n\cdot \binom{n-1}{\ell-1}(k-1)^{n-\ell}.
\]
\end{lemma}
\begin{proof}[Proof of Theorem~\ref{thm:main}, upper bound]
    With hindsight, let $c_0=16$, and let $\ell = \frac{n}{k}(1 + \varepsilon)$ where $\varepsilon = 4\cdot \sqrt{\frac{\ln k}{d}}$.
    Assume $c_0\log k < d$, so that $\varepsilon<1$.
    By Lemma~\ref{lem:gw14}, we have, for all strings $w$ of length $\ell$,
    \begin{align}
        \Pr[X^1,\ldots,X^d \text{ have $w$ as a subsequence}]
        \le \left(\frac{ n\cdot \binom{n}{n-\ell}  (k-1)^{n-\ell}} {k^{n}}\right)^d.
    \end{align}
    By a union bound over all strings of length $\ell$, for $c_k=\frac{k}{4\ln k}$ from Lemma~\ref{lem:h-apx}, we have
    \begin{align}
        \Pr[\LCS(X^1,\ldots,X^d)\ge \ell] 
        &\leq k^{\ell} \cdot \left(\frac{ n\cdot \binom{n}{n-\ell}  (k-1)^{n-\ell}} {k^{n}}\right)^d \nonumber\\
        &\le n^d\cdot k^\ell \cdot \left(\frac{k^{nH_k(1-1/k-\varepsilon/k)}}{k^n} \right)^d \nonumber\\
        &\le n^d\cdot k^\ell \cdot \left(\frac{k^{n(1-c_k(\varepsilon/k)^2)}}{k^n} \right)^d \nonumber\\
        &\le n^d\cdot k^{2n/k} \cdot \left(\frac{k^{n(1-c_k(\varepsilon/k)^2)}}{k^n} \right)^d \nonumber\\
        &= n^d k^{-2n/k} < k^{-n/k}.
    \end{align}
    The second inequality uses Lemma~\ref{lem:binom-apx} and the definition of $\ell$.
    The third inequality uses Lemma~\ref{lem:h-apx}.
    The fourth inequality follows from $\varepsilon<1$.
    The equality follows from plugging in $c_k$.
    Our bound on the expectation follows.
    \begin{align}
        \E[\LCS(X^1,\ldots,X^d)]
        &\le \ell\cdot \Pr[\LCS(X^1,\ldots,X^d)< \ell] + n\cdot \Pr[\LCS(X^1,\ldots,X^d)\ge \ell] \nonumber\\
        &\le \ell\cdot 1 + n\cdot k^{-n/k} \nonumber\\
        &\le \ell + o(1)
    \end{align}
    Taking the limit $n\to\infty$, we conclude $\gamma_{k,d}\le \frac{1+\varepsilon}{k}$, as desired.
\end{proof}

\section*{Acknowledgements}
RL, WR, and YW are supported by NSF grant CCF-2347371.

\bibliographystyle{unsrt}
\bibliography{bib}

\appendix

\section{Binary lower bounds implies $k$-ary}
\label{app:k}

The $k$-ary lower bound in Theorem~\ref{thm:k} follows from the binary lower bound in Theorem~\ref{thm:main} because of the following lemma.
\begin{lemma} \label{lem:main}
    $\gamma_{k,d}\ge\frac{2}{k}\gamma_{2,d}$.
\end{lemma}
\begin{proof}
    Consider $d$ random strings $X^1,\dots,X^d$ over alphabet $[k]$.
    Let $Y^1,\dots,Y^d$ be the subsequences of $X^1,\dots,X^d$ consisting of symbols $\{1,2\}$.
    By standard concentration arguments, the lengths $|Y^1|,\dots,|Y^d|$ are all at least $n_0=\frac{2}{k}n-O_k(n^{2/3})$ with high probability $1-2^{\Omega_k(n^{1/3})}$.
    Conditioned on the lengths $|Y^1|,\dots,|Y^d|$ all being at least $n_0$, the expected LCS of $Y^1,\dots,Y^d$ is at least $\gamma_{2,d}\cdot n_0$.
    Thus, 
    \begin{align}
       \E[\LCS(X^1,\dots,X^d)] 
       &\ge \E[\LCS(Y^1,\dots,Y^d)] \\
       &\ge \E[\LCS(Y^1,\dots,Y^d)\,\mid\,|Y^1|,\dots,|Y^d|\ge n_0] \cdot \Pr[|Y^1|,\dots,|Y^d|\ge n_0]\\
       &\ge \gamma_{2,d}\cdot n_0 \cdot \left(1-2^{-\Omega(n^{1/3})}\right) \\
       &\ge \frac{2}{k}\gamma_{2,d} \cdot n \cdot (1-o(1))
    \end{align}
    Thus, $\gamma_{k,d}\ge \frac{2}{k}\gamma_{2,d}$
\end{proof}

\section{List-decoding against deletions}
 \label{app:deletion}
We connect the generalized Chvátal--Sankoff constant to list-decoding against deletions.
The connection uses Azuma's inequality.
\begin{lemma}[Azuma's inequality]
    Let \(Z_1, Z_2, \cdots, Z_n\) be a martingale with bounded differences, i.e., \(|Z_{i+1}-Z_i| \leq c\) for some constant $c$. Then, for any \(\varepsilon \geq 0\),\[
    \Pr\left( |Z_n - \mathbb{E}[Z_n]| \geq \epsilon \right) \leq 2 \exp \left( -\frac{\epsilon^2}{2 n c^2} \right).
    \]
\end{lemma}
A code is a subset of $[k]^n$.
A random code $C$ is obtained by sampling independent uniformly random strings from $[k]^n$.
For $p\in(0,1)$ and an integer $d\ge 2$, a code $C$ is $(p,d-1)$ list-decodable against deletions if any $d$ strings $X^1,\dots,X^d\in C$ satisfy $\LCS(X^1,\dots,X^d)< (1-p)n$.

The first result in Proposition~\ref{pr:code} says that random codes of positive rate ($|C|\ge 2^{\Omega(n)}$) are list-decodable against deletions with radius $p=1-\gamma_{k,d}-\varepsilon$.
The second result says that random codes even of constant size at not list-decodable against deletions with radius $1-\gamma_{k,d}+\varepsilon$.
Thus, $1-\gamma_{k,d}$ is the maximum fraction of deletions that a positive rate random code list-decodes against with list-size $d$.
\begin{proposition}
    For all $\varepsilon > 0$, there exists a constant $c>0$ such that a random code $C\subset [k]^n$ of size $|C|\ge 2^{cn}$ is $(1-\gamma_{k,d}-\varepsilon,d-1)$ list-decodable against deletions.
    Furthermore, a random code of size $d$, with high probability, not $(1-\gamma_{k,d}+\varepsilon,d-1)$ list-decodable against deletions.
    \label{pr:code}
\end{proposition}
\begin{proof}
    With hindsight, choose $c=\varepsilon^2/10d$.
    We construct the code \(C\) as a set of \( 2^{cn} \) independent random strings, each of length \( n \), drawn from the alphabet \( [k] \). 
    We consider the longest common subsequence (LCS) of \( d \) codewords \( X^1, X^2, \dots, X^d \) from \( C \).
    
    The length of the LCS, LCS\((X^1, X^2, \dots, X^d) \), can be treated as a \textit{martingale sequence} by revealing the symbols one at a time.
   Define \( Z_i \) as the expected value of the LCS length given the first \( i \) symbols of each sequence \( X^1, \dots, X^d \):
   \[
   Z_i = \mathbb{E}[LCS(X^1, \dots, X^d) \mid X^1[1, \dots, i], \dots, X^d[1, \dots, i]].
   \]
   Here, \( Z_0, Z_1, \dots, Z_n \) forms a martingale, where 
   \begin{align}
       Z_0 &= \mathbb{E}[LCS(X^1, \dots, X^d)] \nonumber\\
       Z_n &= LCS(X^1, X^2, \dots, X^d).
   \end{align}
   Further, this martingale has bounded difference $|Z_{i+1}-Z_i|\le 1$.
    By Azuma's inequality, for any \( \epsilon > 0 \), we have:
   \begin{align}
   \Pr\left( |LCS(X^1, X^2, \dots, X^d) - \gamma_{k,d} n| \geq \epsilon n \right)
   = \Pr[|Z_n-Z_0|\ge \varepsilon n] \leq 2 \exp\left( -\frac{\epsilon^2 n}{2} \right).
   \label{eq:azuma}
   \end{align}
   This result implies that, with high probability, the LCS length is close to its expected value \(\gamma_{k,d} n\). With large enough $n$, the probability that LCS exceeds \(\gamma_{k,d} n\) is exponentially small. Thus, for each individual set of \( d \) codewords,
   \[
   \Pr\left( LCS(X^1, X^2, \dots, X^d) > (\gamma_{k,d} + \epsilon) n \right) \leq 2 \exp\left( -\frac{\epsilon^2 n}{2} \right).
   \]
   By the union bound, the probability that any $d$-tuple of codewords in \(C\) violates this bound is at most \[
   |C|^d \cdot 2 \exp\left( -\frac{\epsilon^2 n}{2} \right) \le 2^{-\Omega(n)}.
   \]
   Thus, with high probability, $LCS(X^1, X^2, \dots, X^d) \leq ( \gamma_{k,d} + \epsilon) n$
   for all codewords \( X^1, X^2, \dots, X^d \in C \), and our code is $(1-\gamma_{k,d}-\varepsilon,d-1)$ list-decodable against deletions.

   To show the second result, note that, by \eqref{eq:azuma}, for $d$ independent random strings $X^1,\dots,X^d$
    \[
   \Pr\left( LCS(X^1, X^2, \dots, X^d) > (\gamma_{k,d} - \epsilon) n \right) \geq 1-2 \exp\left( -\frac{\epsilon^2 n}{2} \right).
   \]
   so a random code of size $d$ is \emph{not} $(1-\gamma_{k,d}+\varepsilon,d-1)$ list-decodable with high probability. 
\end{proof}

\end{document}